\documentclass[12pt]{article}
\usepackage{amssymb}
\usepackage{amsthm}
\usepackage{amsmath}
\usepackage{graphicx}
\usepackage{enumerate}
\DeclareMathOperator{\C}{C}
\DeclareMathOperator{\Con}{Con}
\DeclareMathOperator{\BCon}{\mathbf{Con}}
\newtheorem{theorem}{Theorem}[section]
\newtheorem{definition}[theorem]{Definition}
\newtheorem{lemma}[theorem]{Lemma}
\newtheorem{proposition}[theorem]{Proposition}
\newtheorem{remark}[theorem]{Remark}
\newtheorem{example}[theorem]{Example}
\newtheorem{corollary}[theorem]{Corollary}
\title{On ring-like event systems in quantum logic} 
\author{Dietmar Dorninger and Helmut L\"anger}
\date{}
\begin{document}

\maketitle


\begin{abstract} 
A class of ring-like event systems (RLSEs) is studied that generalizes Boolean rings. Quantum logics represented by orthomodular lattices are characterized within this class and the correspondence between Boolean algebras and Boolean rings is enlarged to orthomodular lattices. The structure of RLSEs and various subclasses is analysed and classical logics are especially identified. Moreover, sets of numerical events within different contexts of physical problems are described. A numerical event is defined as a function $p$ from a set $S$ of states of a physical system to $[0,1]$ such that $p(s)$ is the probability of the occurrence of an event when the system is in state $s\in S$. In particular, the question is answered whether a given (small) set of numerical events will give rise to the assumption that one deals with a classical physical system or a quantum mechanical one.  
\end{abstract}



\section{Introduction}

In quantum mechanics so-called quantum logics, also referred to as event systems, are an essential tool for theoretical reasoning and practical computations. The most common event systems are orthomodular lattices (and generalizations of them), in particular, the lattices of closed subspaces of a separable Hilbert space, known as Hilbert logics. Orthomodular lattices can be viewed as a generalization of Boolean algebras, which are characteristic for event systems in classical physics. So the question arises whether the logic behind an experiment might be a Boolean algebra, which, being in one-to-one correspondence to a Boolean ring, can also be understood as a Boolean ring. As often common with electrical engineering, calculations within rings are sometimes preferred to carrying out calculations within lattices. Then the question is whether a logic is a Boolean ring which entails the problem to find an appropriate generalization of Boolean rings corresponding to orthomodular lattices, a problem we will answer in this paper.

For this end we first recall the definition of ring-like structures of events (RLSEs), which we will later on reformulate for a wide class of RLSEs $\mathbf R=(R,+,\cdot,0,1)$ of characteristic $2$ (which means that they satisfy the identity $x+x\approx0$) showing that these structures can be simply obtained by weakening customary axioms of Boolean rings or also by more suggestive other laws.   

\begin{definition}\label{def1}
{\rm(}cf.\ {\rm\cite{DL21})} A {\em ring-like structure of events (RLSE)} is an algebra $(R,+,\cdot,0,1)$ of type $(2,2,0,0)$ such that $(R,\cdot,0,1)$ is a bounded meet-semilattice and satisfying the following identities:
\begin{enumerate}[{\rm(R1)}]
\item $x+y\approx y+x$,
\item $(xy+1)(x+1)+1\approx x$,
\item $\big((xy+1)x+1\big)x\approx xy$,
\item $xy+(x+1)\approx(xy+1)x+1$.
\end{enumerate}
\end{definition}

As one can easily verify, all Boolean rings are RLSEs.

\begin{definition}
An {\rm RLSE} $(R,+,\cdot,0,1)$ is called {\em specific} if it satisfies the identity
\begin{enumerate}
\item[{\rm(R5)}] $x+y\approx x(y+1)+(x+1)y$.
\end{enumerate}
\end{definition}

As we will note below, specific RLSEs are of characteristic $2$ (which, in general, is not the case). Moreover, we will show that there is a one-to-one correspondence between orthomodular lattices and specific RLSEs. This means that the class of all RLSEs is larger than the class of specific RLSEs.

In this paper we will consider various classes of RLSEs, answer the question when an RLSE is a Boolean ring, study structural properties of RLSEs and link RLSEs to so-called {\em algebras of numerical events} (which are sets of probabilities that can be gained by measurements -- cf.\ \cite{BM91} and \cite{MT}). Next we will apply results obtained for RLSEs to algebras of numerical events. We will show that under certain conditions the operation $+$ of RLSEs will coincide with the summation of real functions and the order of the elements of an RLSE with the order of functions. Further, we will give answers to the question, whether a (small) set of numerical events obtained by measurements will justify that one deals with a classical physical system or not. Finally we will weaken the concept of RLSEs by omitting axiom (R3) and associate these structures to sets of numerical events endowed with operations which are relevant for experiments.

To some extent our research is related to the one-to-one correspondence of arbitrary bounded lattices with an antitone involution and so-called pGBQRs (partial generalized Boolean quasirings). -- For further results on pGBQRs cf.\ \cite{BDM} and \cite{DDL10b} -- \cite{DLM01}.

\section{Elementary properties of ring-like structures of \\
events}

Dealing with orthomodular lattices we will denote the supremum of two of its elements $x,y$ by $x\vee y$, their infimum by $x\wedge y$, the complement of an element $x$ by $x'$ and write $x\perp y$ if $x$ and $y$ are orthogonal, i.e.\ if $x\wedge y'=x$. Further we agree to define for $x,y$ of an RLSE $R$, $x\leq y$ if and only if $xy=x$ and to call $x$ and $y$ orthogonal to each other (as ring-like elements), if $x(1+y)=x$.

For every algebra $\mathbf R=(R,+,\cdot,0,1)$ of type $(2,2,0,0)$ let $\mathbb L(\mathbf R)$ denote the algebra $(R,\vee,\wedge,{}',0,1)$ defined by
\begin{align*}
  x\vee y & :=(x+1)(y+1)+1, \\
x\wedge y & :=xy, \\
       x' & :=x+1
\end{align*}
for all $x,y\in R$. As already shown in \cite{DL21} (cf.\ Theorem~2.1) if $\mathbf R$ is an RLSE then $\mathbb L(\mathbf R)$ is an orthomodular lattice. We will likewise use the operations of this lattice within RLSEs. Obviously, the lattice-theoretic orthogonality relation and the one defined above for RLSEs coincide. -- We notice that for RLSEs the orthomodular lattice $\mathbb L(\mathbf R)$ can be a Boolean algebra without $\mathbf R$ being a Boolean ring (cf.\ \cite{DL21}), but this will not be the case with specific RLSEs, as we will show below.

In the following we will make use of the following theorem.

\begin{theorem}\label{th0}
{\rm(}cf.\ {\rm\cite{DL21})} Let $\mathbf R=(R,+,\cdot,0,1)$ be an algebra of type $(2,2,0,0)$. Then $\mathbf R$ is an {\rm RLSE} if and only if $\mathbb L(\mathbf R)=(R,\vee,\wedge,{}',0,1)$ is an orthomodular lattice and $+$ satisfies the following conditions for all $x,y\in R$:
\begin{enumerate}[{\rm(a)}]
\item $x+y=y+x$,
\item $x+1=x'$, 
\item $x+y=x\vee y$ if $x\leq y'$.
\end{enumerate}
\end{theorem}

Further, by means of the lattice structure of RLSEs one can easily see:

\begin{proposition}\label{prop1}
{\rm(}cf.\ {\rm\cite{DL21})} An {\rm RLSE} $\mathbf R=(R,+,\cdot,0,1)$ has the following properties for all $x,y\in R$:
\begin{enumerate}[{\rm(i)}]
\item $(x+1)+1=x$,
\item $x(x+1)=0$, and as a consequence $1+1=0$,
\item $x+0=x$,
\item $x+(x+1)=1$,
\item $x\le y$ if and only if $y+1\le x+1$,
\item $x\perp y$ implies $(x+y)+1=(x+1)(y+1)$,
\item $x+y=x\vee y$ if $x\leq y'$,
\item $x+x=0$ if $\mathbf R$ is specific.
\end{enumerate}
\end{proposition}

Recalling that two elements $x,y$ of an ortholattice are said to {\em commute} (abbreviated by $x\mathrel{\C}y$) if $(x\wedge y)\vee (x\wedge y')= x$ and that $c(x,y):=(x\wedge y)\vee(x\wedge y')\vee(x'\wedge y)\vee(x'\wedge y')$ is called the {\em commutator} of $x$ and $y$, we define analogous concepts for RLSEs $(R,+,\cdot,0,1)$: We say that the elements $x,y$ of $R$ {\em commute} (also indicated by $x\mathrel{\C}y$) if $xy+x(y+1)=x$ and we call the element $c(x,y):=\big(xy+x(y+1)\big)+\big((x+1)y+(x+1)(y+1)\big)$ the {\em commutator} of $x$ and $y$. That these definitions are justified is asserted by the following proposition.

\begin{proposition}\label{prop2}
Let $\mathbf R=(R,+,\cdot,0,1)$ be an {\rm RLSE} and $a,b\in R$. Then the following hold:
\begin{enumerate}
\item[\rm(ix)] $a\mathrel{\C}b$ in $\mathbf R$ if and only if $a\mathrel{\C}b$ in $\mathbb L(\mathbf R)$,
\item[\rm(x)] The commutator of $a$ and $b$ in $\mathbf R$ coincides with the commutator of $a$ and $b$ in $\mathbb L(\mathbf R)$.
\end{enumerate}
\end{proposition}

\begin{proof}
\
\begin{enumerate}
\item[(ix)] Since $a\wedge b\perp a\wedge b'$ we have $(a\wedge b)\vee(a\wedge b')=(a\wedge b)+(a\wedge b')=ab+a(b+1)$. Now $a\mathrel{\C}b$ in $\mathbb L(\mathbf R)$ if and only if $(a\wedge b)\vee(a\wedge b')=a$.
\item[(x)] From (vii) of Proposition~\ref{prop1} we know that $(a\wedge b)\vee(a\wedge b')=ab+a(b+1)$. Replacing $a$ by $a'$ we obtain $(a'\wedge b)\vee(a'\wedge b')=(a+1)b+(a+1)(b+1)$. Since $(a\wedge b)\vee(a\wedge b')\perp(a'\wedge b)\vee(a'\wedge b')$ we then have
\begin{align*}
& (a\wedge b)\vee(a\wedge b')\vee(a'\wedge b)\vee(a'\wedge b')= \\
& =\big((a\wedge b)\vee(a\wedge b')\big)\vee\big((a'\wedge b)\vee(a'\wedge b')\big)= \\
& =\big((a\wedge b)\vee(a\wedge b')\big)+\big((a'\wedge b)\vee(a'\wedge b')\big)= \\
& =(ab+a(b+1))+\big((a+1)b+(a+1)(a+1)\big).
\end{align*}
\end{enumerate}
\end{proof}

For every algebra $\mathbf L=(L,\vee,\wedge,{}',0,1)$ of type $(2,2,1,0,0)$ let $\mathbb R(\mathbf L)$ denote the algebra $(L,+,\cdot,0,1)$ of type $(2,2,0,0)$ defined by
\begin{align*}
x+y & :=(x\wedge y')\vee(x'\wedge y), \\
 xy & :=x\wedge y
\end{align*}
for all $x,y\in L$.

\begin{theorem}\label{th2}
Let $\mathbf R$ be a specific {\rm RLSE} and $\mathbf L$ an orthomodular lattice. Then the following hold:
\begin{enumerate}[{\rm(i)}]
\item $\mathbb L(\mathbf R)$ is an orthomodular lattice,
\item $\mathbb R(\mathbf L)$ is a specific {\rm RLSE},
\item $\mathbb R\big(\mathbb L(\mathbf R)\big)=\mathbf R$,
\item $\mathbb L\big(\mathbb R(\mathbf L)\big)=\mathbf L$.
\end{enumerate}
\end{theorem}

\begin{proof}
Let
\begin{align*}
                              \mathbf R & =(R,+,\cdot,0,1), \\
                   \mathbb L(\mathbf R) & =(R,\vee,\wedge,{}',0,1), \\
\mathbb R\big(\mathbb L(\mathbf R)\big) & =(R,\oplus,\odot,0,1), \\
                              \mathbf L & =(L,\vee,\wedge,{}',0,1), \\
                   \mathbb R(\mathbf L) & =(L,+,\cdot,0,1), \\
\mathbb R\big(\mathbb L(\mathbf R)\big) & =(L,\cup,\cap,{}^*,0,1).
\end{align*}
\begin{enumerate}
\item[(i)] follows from Theorem~\ref{th0}.
\item[(iii)] Using Proposition~\ref{prop1} (vii) we get
\begin{align*}
x\oplus y & \approx(x\wedge y')\vee(x'\wedge y)\approx x(y+1)+(x+1)y\approx x+y, \\
 x\odot y & \approx x\wedge y\approx xy.
\end{align*}
\item[(iv)] We have
\begin{align*}
    x+1 & \approx(x\wedge1')\vee(x'\wedge1)\approx x', \\
x\cup y & \approx(x+1)(y+1)+1\approx(x'\wedge y')'\approx x\vee y, \\
x\cap y & \approx xy\approx x\wedge y, \\
    x^* & \approx x+1\approx x'.
\end{align*}
\item[(ii)] follows from Theorem~\ref{th0} and (iv) since for all $x,y\in L$\begin{enumerate}[(a)]
\item $x+y=(x\wedge y')\vee(x'\wedge y)=(y\wedge x')\vee(y'\wedge x)=y+x$,
\item $x+1=(x\wedge1')\vee(x'\wedge1)=x'$,
\item $x+y=(x\wedge y')\vee(x'\wedge y)=x\vee y$ if $x\leq y'$.
\end{enumerate}
\end{enumerate}
\end{proof}

\begin{corollary}
For fixed base set $A$, the mappings $\mathbb L$ and $\mathbb R$ are mutually inverse bijections between the set of all specific {\rm RLSEs} over $A$ and the set of all orthomodular lattices over $A$.
\end{corollary}

\begin{corollary}\label{cor1}
For a specific {\rm RLSE} $\mathbf R$ the associated orthomodular lattice $\mathbb L(\mathbf R)$ is a Boolean algebra if and only if $\mathbf R$ is a Boolean ring.
\end{corollary}

\begin{corollary}\label{cor2}
For a specific {\rm RLSE} $(R,+,\cdot,0,1)$ the condition $x(y+1)=xy+x$ for some $x,y\in R$ is equivalent to $x\mathrel{\C}y$.
\end{corollary}

\begin{proof}
The equation $x(y+1)=xy+x$ is equivalent to $x\wedge(x'\vee y')=x\wedge y'$ which according to results in \cite K is equivalent to $x\mathrel{\C}y'$ and hence to $x\mathrel{\C}y$.
\end{proof}

\begin{corollary}\label{cor3}
A specific {\rm RLSE} $(R,+,\cdot,0,1)$ is a Boolean ring if and only if it satisfies the identity $x(y+1)\approx xy+x$.
\end{corollary}

\begin{proof}
This follows from Corollary~\ref{cor2} and from the fact that an orthomodular lattice $(L,\vee,\wedge,{}',0,1)$ is a Boolean algebra if and only if $x\mathrel{\C}y$ for all $x,y\in L$ (cf.\ \cite K).
\end{proof}

\section{Structure theory of RLSEs}

\begin{definition}\label{def2}
An {\rm RLSE} $(R,+,\cdot,0,1)$ is called {\em weakly distributive} if it satisfies the identity
\begin{enumerate}
\item[{\rm(R6)}] $(xy+1)x\approx xy+x$.
\end{enumerate}
\end{definition}

Obviously, any specific RLSE $(R,+,\cdot,0,1)$ is weakly distributive, because according to Proposition~\ref{prop1} (ii) and (R5) we have
\[
(xy+1)x\approx xy(x+1)+(xy+1)x\approx xy+x.
\]
Moreover, any weakly distributive RLSE is of characteristic $2$ since
\[
x+x\approx x1+x\approx(x1+1)x\approx(x+1)x\approx0
\]
according to Proposition~\ref{prop1} (ii).

\begin{example}
The specific {\rm RLSE} corresponding to the orthomodular lattice $\mathbf{MO}_2$ is weak\-ly distributive.
\end{example}

According to Corollary~\ref{cor1} this example shows that in general a specific RLSE $\mathbf R$ is not a Boolean ring.

The next theorem explains how the initially introduced axioms for RLSEs can be rephrased by weakening the customary axioms of associativity and distributivity known from Boolean rings in case of weakly distributive RLSEs.
 
\begin{theorem}\label{th4}
Let $\mathbf R=(R,+,\cdot,0,1)$ be an algebra of type $(2,2,0,0)$ such that $(R,\cdot,0,1)$ is a bounded meet-semilattice. Then the following are equivalent:
\begin{enumerate}[{\rm(i)}]
\item $\mathbf R$ is a weakly distributive {\rm RLSE},
\item $\mathbf R$ satisfies the following identities:
\begin{enumerate}[{\rm(W1)}]
\item $0+1\approx 1$,
\item $x+y\approx y+x$,
\item $(xy+x)+1\approx xy+(x+1)$,
\item $(xy+x)+x\approx xy+(x+x)$,
\item $(xy+1)x\approx xy+x$,
\item $(xy+1)(x+1)\approx xy(x+1)+(x+1)$.
\end{enumerate}
\end{enumerate}
\end{theorem}

\begin{proof}
$\text{}$ \\
(i) $\Rightarrow$ (ii):
\begin{enumerate}[(W1)]
\item follows from Proposition~\ref{prop1} (iii) and (R1).
\item equals (R1).
\item We have
\[
(xy+x)+1\approx(xy+1)x+1\approx xy+(x+1)
\]
according to (R6) and (R4).
\item We find
\[
(xy+x)+x\approx(xy+1)x+x\approx\big((xy+1)x+1\big)x\approx xy\approx xy+0\approx xy+(x+x)
\]
according to (R6), (R3), Proposition~\ref{prop1} (iii) and the fact that every weakly distributive RLSE is of characteristic $2$.
\item equals (R6).
\item We notice
\[
(xy+1)(x+1)\approx\big((xy+1)(x+1)+1\big)+1\approx x+1\approx0+(x+1)\approx xy(x+1)+(x+1)
\]
according to Proposition~\ref{prop1} (i), (R2), Proposition~\ref{prop1} (iii), (R1) and Proposition~\ref{prop1} (ii).
\end{enumerate}
(ii) $\Rightarrow$ (i): \\
Putting $y=0$ in (W5) and using (W1) and (W2) yields $x+0\approx0+x\approx x$. \\
Putting $y=0$ in (W4) yields $x+x\approx0$. \\
Setting $y=1$ in (W5) we get $x(x+1)\approx0$. \\
Setting $x=1$ in (W4) we obtain $(y+1)+1\approx y$. \\
In the sequel we often use these identities without mentioning them.
\begin{enumerate}
\item[(R1)] equals (W2).
\item[(R2)] We have
\[
(xy+1)(x+1)+1\approx\big(xy(x+1)+(x+1)\big)+1\approx\big(0+(x+1)\big)+1\approx(x+1)+1\approx x
\]
due to (W6).
\item[(R3)] We get
\[
\big((xy+1)x+1\big)x\approx(xy+1)x+x\approx(xy+x)+x\approx xy
\]
according to (W5) and (W4).
\item[(R4)] We see that
\[
xy+(x+1)\approx(xy+x)+1\approx(xy+1)x+1
\]
accordingly to (W3) and (W5).
\item[(R6)] equals (W5).
\end{enumerate}
\end{proof}

Recalling that for an RLSE $\mathbf R=(R,+,\cdot,0,1)$ $a\leq b$ for $a,b\in R$ means that $ab=a$, a notion which coincides with $a\leq b$ for the associated lattice $\mathbb L(\mathbf R)$, we have $\{(xy,x)\mid x,y\in R\}=\{(x,y)\in R^2\mid x\leq y\}$ and we can now rephrase Theorem~\ref{th4} as follows:

\begin{theorem}\label{th5}
Let $\mathbf R=(R,+,\cdot,0,1)$ be an algebra of type $(2,2,0,0)$ such that $(R,\cdot,0,1)$ is a bounded meet-semilattice. Then the following are equivalent:
\begin{enumerate}[{\rm(i)}]
\item $\mathbf R$ is a weakly distributive {\rm RLSE},
\item $\mathbf R$ satisfies the following identities and conditions:
\begin{enumerate}[{\rm(1)}]
\item $0+1\approx 1$,
\item $x+y\approx y+x$,
\item if $x\leq y$ then $(x+y)+1=x+(y+1)$,
\item if $x\leq y$ then $(x+y)+y=x+(y+y)$,
\item if $x\leq y$ then $(x+1)y=x+y$,
\item if $x\leq y$ then $(x+1)(y+1)=x(y+1)+(y+1)$.
\end{enumerate}
\end{enumerate}
\end{theorem}

The identities (W3) and (W4) in Theorem~\ref{th4} are special cases of associativity. A further version of associativity is the following.

\begin{definition}
{\rm(}cf.\ {\rm\cite{DL21})} An {\rm RLSE} $(R,+,\cdot,0,1)$ is called {\em weakly associative} if it satisfies the identity
\begin{enumerate}
\item[{\rm(R7)}] $(x+y)+1\approx x+(y+1)$.
\end{enumerate}
\end{definition}

Of course, every Boolean ring is weakly associative. The converse does not hold. The RLSE $\mathbf R$ of characteristic $2$ with $\mathbb L(\mathbf R)=\mathbf{MO}_2$, $a+b=a'+b'=c$ and $a+b'=a'+b=c'$ for $a\neq b$, $0<a,b<1$ and an arbitrary $c\in{\rm MO}_2$ is weakly associative, but not a Boolean ring. As shown in \cite{DL21} a weakly associative specific RLSE is a Boolean ring. We now note that a weakly associative RLSE is weakly distributive (and hence of characteristic $2$) since
\[
(xy+1)x\approx\big((xy+1)x+1\big)+1\approx\big(xy+(x+1)\big)+1\approx\big((xy+x)+1\big)+1\approx xy+x
\]
according to Proposition~\ref{prop1} (i), (R4) and weak associativity. The converse does not hold. The specific RLSE corresponding to the orthomodular lattice $\mathbf{MO}_2$ is weakly distributive, but not weakly associative since for two incomparable elements $a$ and $b$
\begin{align*}
(a+b)+1 & =\big((a\wedge b')\vee(a'\wedge b)\big)'=(0\vee0)'=0'=1\neq0=0\vee0=(a\wedge b)\vee(a'\wedge b')= \\
        & =a+b'=a+(b+1).
\end{align*}

We close this section with a purely algebraic remark about the structure of RLSEs.

Let $\mathbf A=(A,F)$ be an algebra. Then by $\Con\mathbf A$ we denote the set of all congruences on $\mathbf A$ and by $\BCon\mathbf A=(\Con\mathbf A,\subseteq)$ the congruence lattice of $\mathbf A$. The algebra $\mathbf A$ is called
\begin{itemize}
\item {\em congruence permutable} if $\Theta\circ\Phi=\Phi\circ\Theta$ for all $\Theta,\Phi\in\Con\mathbf A$,
\item {\em congruence distributive} if $\BCon\mathbf A$ is distributive,
\item {\em arithmetical} if it is both congruence permutable and congruence distributive,
\item {\em congruence regular} if for all $a\in A$ and $\Theta,\Phi\in\Con\mathbf A$, $[a]\Theta=[a]\Phi$ implies $\Theta=\Phi$,
\item {\em congruence uniform} if for every $\Theta\in\Con\mathbf A$ all classes of $\Theta$ have the same cardinality.
\end{itemize}

\begin{remark}
{\rm RLSEs} are arithmetical, congruence regular and congruence uniform.
\end{remark}

\begin{proof}
Let $\mathbf R$ be an RLSE. Since the fundamental operations of $\mathbb L(\mathbf R)$ are terms in $\mathbf R$ we have $\Con\mathbf R\subseteq\Con\mathbb L(\mathbf R)$. Now the theorem follows from the fact (see e.g.\ \cite{CEL}) that orthomodular lattices are arithmetical, congruence regular and congruence uniform.
\end{proof}

\section{Algebras of numerical events}

Let $S$ be a set of states of a physical system and $p(s)$ the probability of the occurrence of an event when the system is in state $s\in S$. The function $p$ from $S$ to $[0,1]$ is called a {\em numerical event}, or more precisely, an {\em S-probability} (cf.\ \cite{BM91} and \cite{BM93}).

Let $P$ be a set of S-probabilities including the constant functions $0$ and $1$. We denote the order of real functions by $\leq$, write $p':=1-p$ for the counter probability of $p\in P$ and $p\perp q$ if $p$ and $q$ are orthogonal in $P$, i.e.\ $p\leq q'$. If the infimum or supremum of $p,q\in P$ exists in $P$, we denote this by $p\wedge q$ and $p\vee q$, respectively. Finally we agree to write $p+q$, $p-q$ and $pq$ for the sum, difference and product of functions $p,q\in P$. Not to mix up the sum and product of functions with the sum and product within RLSEs, we will use with RLSEs $\oplus$ and $\odot$ instead of $+$ and $\cdot$, respectively.

\begin{definition}{\rm(}cf.\ {\rm\cite{BM91})}\label{def3}
A set $P$ of S-probabilities is called an {\em algebra of S-probabilities} if
\begin{enumerate}[{\rm(S1)}]
\item $0,1\in P$,
\item $p'\in P$ for every $p\in P$,
\item if $p\perp q\perp r\perp p$ for $p,q,r\in P$ then $p+q+r\in P$.
\end{enumerate}
\end{definition}

Putting $r=0$ in axiom (S3) one obtains that $p\perp q$ implies $p+q\in P$ in which case one can show that $p+q=p\vee q$ (in respect to the order $\leq$ of the functions of $P$).

In general, $(P,\vee,\wedge,{}',0,1)$ is an orthomodular poset in respect to the partial order of functions, but from now on we will assume with good cause that $P$ is a lattice.

That an algebra of S-probabilities is a lattice and hence an orthomodular lattice, is a typical feature of many quantum logics. In particular, every Hilbert-space logic can be considered as a lattice-ordered algebra of S-probabilities (cf.\ \cite{BM93}), and in the important case that $|S|=2$ every algebra of S-probabilities is a lattice (cf.\ \cite{DDL10a}). Of course, also all classical logics whose order $\leq$ correspond to Boolean algebras (cf.\ \cite{MT})can be understood as lattice-ordered algebras of numerical events.

If measurements are available in the context of a set $P_n$ of numerical events it is often crucial to get to know whether one deals with a classical situation or a quantum-mechanical one which means that one has to decide whether $P_n$ can be embedded into an algebra of S-probabilities $P$ in such a way that the elements of $P_n$ lie within a Boolean subalgebra of $P$. If this is the case, $P_n$ is called {\em Boolean embeddable}, or for short, only {\em embeddable} (cf.\ \cite{DLM20}).

Let $P$ be a lattice-ordered algebra of S-probabilities and
\begin{align*}
p\oplus q & :=(p\wedge q')\vee(p'\wedge q), \\
 p\odot q & :=p\wedge q
\end{align*}
for all $p,q\in P$. Then, as shown in \cite{DL21}, $\mathbf R=(P,\oplus,\odot,0,1)$ is an RLSE with $\mathbb L(\mathbf R)=P$. We call $\mathbf R$ the {\em {\rm RLSE} associated to P}. $\mathbf R$ has characteristic $2$ and is weakly distributive, because $\mathbf R$ obviously satisfies identity (R5). If a set $P_n$ of numerical events is Boolean embeddable into $P$ we will also say that it is Boolean embeddable into $\mathbf R$.

Next we express $\oplus$, $\odot$ and $\leq$ of RLSEs associated to algebras of S-probabilities by the sum, difference and $\leq$ of real functions.

\begin{proposition}\label{prop3}
Let $\mathbf R$ be the RLSE associated to a lattice-ordered algebra of S-prob\-a\-bil\-i\-ties $P$ and $p,q\in P$. Then
\begin{enumerate}[{\rm(i)}]
\item $p\oplus q=p\odot(1-q)+(1-p)\odot q$,
\item $p\oplus q=q-p$ if $p\leq q$,
\item $p\oplus1=1-p$,
\item $p\oplus q=p+q$ if $p\perp q$.
\end{enumerate}
\end{proposition}

\begin{proof}
\
\begin{enumerate}[(i)]
\item holds since the operations $\odot$ and $\wedge$ coincide, $p'=1-p$ and $(p\odot q')\perp(p'\odot q)$.
\item is a consequence of Theorem \ref{th5} and Proposition 2.1 in \cite{DDL10a}.
\item follows from (ii).
\item results from Proposition~\ref{prop1} (vii).
\end{enumerate}
\end{proof}

When we say that a set $P_n$ of S-probabilities is Boolean embeddable into an RLSE $\mathbf R=(R,\oplus,\odot,0,1)$ we assume that there exists an (arbitrary) lattice-ordered algebra $P$ of S-probabilities such that $P=\mathbb L(\mathbf R)$. If the elements of $P_n$ can only have two values, namely $0$ and $1$, then we also assume this for the elements of $P$. Such an algebra of S-probabilities then is a so-called concrete logic, that is a quantum logic which can be represented by sets.

\begin{theorem}\label{th6}
Let $P_2=\{p,q\}$. Then the following holds:
\begin{enumerate}[{\rm(i)}]
\item $P_2$ is Boolean embeddable if and only if $p\odot(1-q)=p-p\odot q$.
\item If $p$ and $q$ are two-valued then $P_2$ is Boolean if and only if $p\odot q=pq$.
\end{enumerate}
\end{theorem}

\begin{proof}
\
\begin{enumerate}[(i)]
\item $p$, $q$ are Boolean embeddable into an orthomodular lattice if and only if $p\mathrel{\C}q$. According to Corollary~\ref{cor2} this is equivalent to $p\odot(q\oplus1)=p\odot q\oplus p$ which by Proposition~\ref{prop3} means $p\odot(1-q)=p-p\odot q$.
\item We assume that $p$ and $q$ can only have the values $0$ and $1$. For orthomodular lattices $p\mathrel{\C}q$ is equivalent to $p\mathrel{\C}q'$, hence by what we have already proved $p$ and $q$ are Boolean embeddable if and only if $p\odot q=p-p\odot(1-q)$. Within a Boolean subalgebra $p(s)\wedge q(s)=\min\big(p(s),q(s)\big)=p(s)q(s)$ with $\min$ short for minimum. Hence $pq$ must be the element $p\odot q$ of $\mathbf R$. Conversely, if $pq=p\odot q$ then $p\odot q\oplus p=p-p\odot q=p-pq=p(1-q)=p\odot(1-q)$ since $p(1-q)$ is an element of $\mathbf R$ that coincides with $p\wedge(1-q)$.
\end{enumerate}
\end{proof} 

Let $A$ be a finite subset of an RLSE $\mathbf R$. We denote by $\prod_{\bf R}A$ the product in $\mathbf R$ of all elements of $A$ and by $\bigwedge A$ the infimum of these elements in $\mathbb L(\mathbf R)$. Moreover, we will denote the product within the reals of all functions belonging to $A$ by $\prod A$ and the set-theoretic union of $A$ and $B$ by $A\cup B$. As proven in \cite{DL14} for $n>1$ an $n$-element subset $T$ of an orthomodular poset is Boolean embeddable if and only if $\bigwedge A$ and $\bigwedge B$ commute for every $k\in\{1,\ldots,n-1\}$ and every $k$-element subset $A$ and $B$ of $T$. Taking this into account one can derive from Theorem~\ref{th6} a rough procedure to find out whether a set $P_n=\{p_1,\ldots,p_n\}$ of S-probabilities is Boolean embeddable, namely

For $k=1$ to $n-1$: \\
Check for every $k$-element subsets $A$ and $B$ of $P_n$ whether $\prod_{\bf R}A\odot(1-\prod_{\bf R}B)=\prod_{\bf R}A-\prod_{\bf R}A\odot\prod_{\bf R}B$, or rather whether $\prod_{\bf R}A\odot\prod_{\bf R}B=\prod(A\cup B)$, if all S-probabilities are two-valued.

We conclude this paragraph by weakening two former concepts.

\begin{definition}
Omitting axioms {\rm(R3)} and {\rm(R4)} in the definition of an {\rm RLSE} we will call the structure arising this way a {\em near-RLSE}, and if a {\em near-RLSE} satisfies axiom {\rm(R5)} we will call it {\em specific}. Moreover, substituting axiom {\rm(S3)} in the definition of algebras of S-probabilities by its special case $r=0$, i.e.\ if $p\perp q$ for $p,q\in P$ then $p+q\in P$, we obtain a so-called {\em generalized field of events} {\rm(GFE)} {\rm(}cf.\ {\rm\cite D)}.
\end{definition}

If $\mathbf R=(R,\oplus,\odot,0,1)$ is a near-RLSE then $\mathbb L(\mathbf R)$ is a lattice with an antitone involution $'$ that in general is not a complementation. Such a lattice could be considered as a quantum logic, however, we will now focus on a different approach to near-RLSEs.

We consider a set $Q$ of S-probabilities containing $0$ and $1$ endowed by the operations $\oplus$ and $\odot$ defined for $p,q\in Q$ by 
\begin{align*}
(p\oplus q)(s) & :=\max\big(p(s),q(s)\big)-\min\big(p(s),q(s)\big), \\
 (p\odot q)(s) & :=\min\big(p(s),q(s)\big)
\end{align*}
for all $s\in S$ with $\min$ and $\max$ having the obvious meanings.

As for these operations one could think of repeating an experiment several times for the same states $s\in S$, $p\oplus q$ giving the bandwidth between the lowest and highest values of two repetitions $p$ and $q$, $p\odot q$ the obtained lowest values and $p\oplus1=1-p$ providing the counter probability to $p$.

In the proof of the next theorem we use the following two lemmas.

\begin{lemma}\label{lem1}
Let $p,q\in Q$ and $s\in S$ and assume $p(s),q(s)\in\{0,1\}$. Then
\begin{align*}
p(s)\oplus q(s) & =p(s)+q(s)-2p(s)q(s), \\
 p(s)\odot q(s) & =p(s)q(s).
\end{align*}
\end{lemma}

\begin{proof}
Consider the four cases $\big(p(s),q(s)\big)\in\{0,1\}^2$.
\end{proof}

\begin{lemma}\label{lem2}
Put
\begin{align*}
x\oplus y & :=\max(x,y)-\min(x,y), \\
 x\odot y & :=\min(x,y)
\end{align*}
for all $x,y\in[0,1]$ and let $a,b\in[0,1]$. Then the following holds:
\begin{enumerate}[{\rm(i)}]
\item $a\oplus b=b\oplus a$,
\item $(a\odot b\oplus1)\odot(a\oplus1)\oplus1=a$,
\item $\big((a\odot b\oplus1)\odot a\oplus1\big)\odot a=a\odot b$ if and only if $b\geq\min(a,1-a)$,
\item $a\odot b\oplus(a\oplus1)=(a\odot b\oplus1)\odot a\oplus1$ if and only if $a\in\{0,1\}$ or $b=0$,
\item $a\oplus b=a\odot(b\oplus1)\oplus(a\oplus1)\odot b$.
\end{enumerate}
\end{lemma}

\begin{remark}
The equalities in {\rm(i)} -- {\rm(v)} correspond exactly to the identities {\rm(R1)} -- {\rm(R5)}.
\end{remark}

\begin{proof}[Proof of Lemma~\ref{lem2}]
\
\begin{enumerate}[(i)]
\item $a\oplus b=\max(a,b)-\min(a,b)=\max(b,a)-\min(b,a)=b\oplus a$
\item $(a\odot b\oplus1)\odot(a\oplus1)\oplus1=1-\min\big(1-\min(a,b),1-a\big)=\max(\min(a,b),a)=a$
\item If $b\geq a$ then
\begin{align*}
\big((a\odot b\oplus1)\odot a\oplus1\big)\odot a & =\min\big(1-\min(1-a,a),a\big)=\min\big(\max(a,1-a),a\big)= \\
                                                 & =a=\min(a,b)=a\odot b.
\end{align*}
If $b<a$ and $b\geq1-a$ then
\[
\big((a\odot b\oplus1)\odot a\oplus1\big)\odot a=\min(b,a)=\min(a,b)=a\odot b.
\]
If, finally, $b<a$ and $b<1-a$ then
\[
\big((a\odot b\oplus1)\odot a\oplus1\big)\odot a=\min(1-a,a)>b=\min(a,b)=a\odot b.
\]
\item Because of
\begin{align*}
a\odot b\oplus(a\oplus1) & =\max\big(\min(a,b),1-a\big)-\min\big(\min(a,b),1-a\big)= \\
                         & =\max\big(\min(a,b),1-a\big)-\min(a,b,1-a)
\end{align*}
and
\[
(a\odot b\oplus1)\odot a\oplus1=1-\min\big(1-\min(a,b),a\big)=\max\big(\min(a,b),1-a\big),
\]
$a\odot b\oplus(a\oplus1)=(a\odot b\oplus1)\odot a\oplus1$ if and only if $\min(a,b,1-a)=0$ which means $a\in\{0,1\}$ or $b=0$.
\item If $a\leq1-b$ then
\begin{align*}
a\oplus b & =\max(a,b)-\min(a,b)= \\
          & =\max\big(\min(a,1-b),\min(1-a,b)\big)-\min\big(\min(a,1-b),\min(1-a,b)\big)= \\
					& =a\odot(b\oplus1)\oplus(a\oplus1)\odot b.
\end{align*}
If $a\geq1-b$ then
\begin{align*}
a\oplus b & =\max(a,b)-\min(a,b)=1-\min(a,b)-\big(1-\max(a,b)\big)= \\
          & =\max(1-b,1-a)-\min(1-b,1-a)= \\
          & =\max\big(\min(a,1-b),\min(1-a,b)\big)-\min\big(\min(a,1-b),\min(1-a,b)\big)= \\
          & =a\odot(b\oplus1)\oplus(a\oplus1)\odot b.
\end{align*}
\end{enumerate}
\end{proof}

Now we can prove our final result.

\begin{theorem}
For $\mathbf Q=(Q,\oplus,\odot,0,1)$ the following hold:
\begin{enumerate}[{\rm(i)}]
\item $\mathbf Q$ is a specific near-{\rm RLSE},
\item $\mathbf Q$ is a {\rm GFE} in respect to the order $\leq$ of functions,
\item the following are equivalent:
\begin{enumerate}[{\rm(a)}]
\item $Q\subseteq\{0,1\}^S$,
\item $\mathbf Q$ satisfies identity {\rm(R3)},
\item $\mathbf Q$ satisfies identity {\rm(R4)},
\item $\mathbf Q$ is an {\rm RLSE},
\item $\mathbf Q$ is a Boolean ring.
\end{enumerate}
\end{enumerate}
\end{theorem}

\begin{proof}
Let $p,q\in Q$.
\begin{enumerate}[(i)]
\item follows immediately from Lemma~\ref{lem2}.
\item If $p\perp q$ (in respect to $\leq$, not in the sense of RLSEs) then $p\leq1-q$ and hence
\[
p+q=1-(1-q-p)=1-\big(\max(p,1-q)-\min(p,1-q)\big)=\big(p\oplus(q\oplus1)\big)\oplus1\in Q.
\]
\item The equivalence of (a) -- (d) can be directly derived from Lemma~\ref{lem2}. \\
\big((a) and (d)\big) $\Rightarrow$ (e): \\
Using Lemma~\ref{lem1} we obtain
\[
p\odot(q\oplus1)=p(1-q)=p-pq=pq+p-2pqp=p\odot q\oplus p
\]
which by Corollary~\ref{cor3} is equivalent to (e). \\
(e) $\Rightarrow$ (d): \\
This is already well-known.
\end{enumerate}
\end{proof}

Authors' addresses:

Dietmar Dorninger \\
TU Wien \\
Faculty of Mathematics and Geoinformation \\
Institute of Discrete Mathematics and Geometry \\
Wiedner Hauptstra\ss e 8-10 \\
1040 Vienna \\
Austria \\
dietmar.dorninger@tuwien.ac.at

Helmut L\"anger \\
TU Wien \\
Faculty of Mathematics and Geoinformation \\
Institute of Discrete Mathematics and Geometry \\
Wiedner Hauptstra\ss e 8-10 \\
1040 Vienna \\
Austria, and \\
Palack\'y University Olomouc \\
Faculty of Science \\
Department of Algebra and Geometry \\
17.\ listopadu 12 \\
771 46 Olomouc \\
Czech Republic \\
helmut.laenger@tuwien.ac.at

\begin{thebibliography}{99}
\bibitem{BDM}
E.~G.~Beltrametti, D.~Dorninger and M.~J.~M\c aczy\'nski, On a cryptographical characterization of classical and nonclassical event systems.  Rep.\ Math.\ Phys.\ {\bf60} (2007), 117--123. 
\bibitem{BM91}
E.~G.~Beltrametti and M.~J.~M\c aczy\'nski, On a characterization of classical and nonclassical probabilities. J.\ Math.\ Phys.\ {\bf32} (1991), 1280--1286.
\bibitem{BM93}
E.~G.~Beltrametti and M.~J.~M\c aczy\'nski, On the characterization of probabilities: a generalization of Bell's inequalities. J.\ Math.\ Phys.\ {\bf34} (1993), 4919--4929.
\bibitem{CEL}
I.~Chajda, G.~Eigenthaler and H.~L\"anger, Congruence Classes in Universal Algebra. Heldermann, Lemgo 2012. ISBN 3-88538-226-1.
\bibitem{DDL10a}
G.~Dorfer, D.~Dorninger and H.~L\"anger, On the structure of numerical event spaces. Kybernetica {\bf46} (2010), 971--981.
\bibitem{DDL10b}
G.~Dorfer, D.~Dorninger and H.~L\"anger, On algebras of multidimensional probabilities. Math.\ Slovaca {\bf60} (2010), 571--582.
\bibitem D
D.~Dorninger, On the structure of generalized fields of events. Contr.\ General Algebra {\bf20} (2012), 29--34.
\bibitem{DL14}
D.~Dorninger and H.~L\"anger, A note on Boolean subsets of orthomodular posets. Ital.\ J.\ Pure Appl.\ Math.\ {\bf32} (2014), 277--282.
\bibitem{DL21}
D.~Dorninger and H.~L\"anger, On ring-like structures of lattice-ordered numerical events. Asian-Eur.\ J.\ Math.\ {\bf14} (2021), 2150186-1 - 2150186-10.
\bibitem{DLM01}
D.~Dorninger, H.~L\"anger and M.~M\c aczy\'nski, Ring-like structures with unique symmetric difference related to quantum logic. Discuss.\ Math.\ Gen.\ Algebra Appl.\ {\bf21} (2001), 239--253. 
\bibitem{DLM20} 
D.~Dorninger, H.~L\"anger and M.~J.~M\c aczy\'nski, Boolean properties and Bell-like inequalities of numerical events. Rep.\ Mat.\ Phys.\ {\bf85} (2020), 147--162. 
\bibitem K
G.~Kalmbach, Orthomodular Lattices. Academic Press, London 1983. ISBN 0-12-394580-1.
\bibitem{MT}
M.~J.~M\c aczy\'nski and T.~Traczyk, A characterization of orthomodular partially ordered sets admitting a full set of states. Bull.\ Acad.\ Polon.\ Sci.\ S\'er.\ Sci.\ Math.\ Astronom.\ Phys.\ {\bf21} (1973), 3--8.
\end{thebibliography}
\end{document}